\documentclass[12pt]{article}
\usepackage{amsmath, amsthm, amssymb,latexsym, enumerate}
\usepackage{graphicx}
\usepackage{hyperref}
\usepackage{xcolor}
\usepackage{cite}
\usepackage{dsfont}
\usepackage{graphicx, caption, subcaption}

\numberwithin{equation}{section}
\newtheorem{theorem}{Theorem}[section]

\newtheorem{corollary}[theorem]{Corollary}
\newtheorem{problem}{Problem}
\newtheorem{claim}[theorem]{Claim}
\newtheorem{subclaim}[theorem]{Subclaim}

\newtheorem{remark}[theorem]{Remark}
\newtheorem{proposition}[theorem]{Proposition}
\title{Tight upper bounds on the hop domination number of triangle-free graphs}
\author{Shinya Fujita\\[1ex]
\small School of Data Science, Yokohama City University,\\
\small Yokohama 236-0027, Japan\\
\small \tt fujita@yokohama-cu.ac.jp\\
[2.5ex]
Boram Park\\[1ex]
\small Department of Mathematics, Ajou University,\\
\small Suwon 16499, Republic of Korea \\
\small\tt borampark@ajou.ac.kr
}

\date{\today}

\begin{document}

\maketitle

\begin{abstract}
For a graph $G$, a subset $S$ of $V(G)$ is a {\it hop dominating set} of $G$ if every vertex not in $S$ has a $2$-step neighbor in $S$. The {\it hop domination number}, $\gamma_h(G)$, of $G$ is the minimum cardinality
of a hop dominating set of $G$.
In this paper, we show that for a connected triangle-free graph $G$ with $n\ge 15$ vertices, 
if $\delta(G)\ge 2$, then $\gamma_h(G)\le \frac{2n}{5}$, and the bound is tight.
We also give some tight upper bounds on $\gamma_h(G)$ for {triangle-free} graphs $G$ that contain a Hamiltonian path or a Hamiltonian cycle. 
\\


\end{abstract} 

\section{Introduction} 
Let $G$ be a graph.
A subset $S$ of $V(G)$ is a \textit{dominating set} of $G$ if every vertex in $V(G)\setminus S$ has a neighbor in $S$. The
\textit{domination number} $\gamma(G)$ of $G$ is the minimum cardinality of a dominating set of $G$. 
A subset $S$ of $V(G)$ is a \textit{total dominating set} of $G$ if every vertex in $V(G)$ has a neighbor in $S$. 
The \textit{total domination number}, $\gamma_t(G)$, is the minimum cardinality of a total dominating set of $G$. 
The dominating set problem in graph theory has been a topic of interest for many researchers and is related to network coverage and control problems, applied in various fields like communication networks, social networks, and more (see \cite{book1,book2,book3,book4}).

For a graph $G$, a subset $S$ of $V(G)$ is a {\it hop dominating set} of $G$ if for any vertex $v\in V(G)\setminus S$, there exists a vertex $u_v\in S$ such that the distance between $u_v$ and $v$ in $G$ is exactly $2$. 
The {\it hop domination number}, $\gamma_h(G)$, of $G$ is the minimum cardinality
of a hop dominating set of $G$. By definition, note that $\gamma_h(K_n)=n$ holds for the complete graph $K_n$.
The concept of hop dominating set was originally introduced by Natarajan and Ayyaswamy in \cite{natarajan2015hop}. Henning and Rad \cite{henning20172} further explored this concept: Indeed, they showed that a connected graph $G$ of order $n\geq 3$ satisfies $\gamma_h(G)=n-1$ if and only if $G\cong K_n^-$ (that is, the graph obtained from $K_n$ by deleting one edge), thereby answering a question posed by Natarajan and Ayyaswamy in \cite{natarajan2015hop}; and moreover, they gave probabilistic upper bounds for the hop domination number of a graph and also showed that the decision problems on the hop dominating set problems are NP-complete for planar bipartite graphs and planar chordal graphs. More computational results were given by Henning, Pal, Pradhan \cite{henning2020graphs} along this line. 
Furthermore, they found an important relationship between the hop domination number and the total domination number of a triangle-free graph:

\begin{theorem}\cite{henning20172}\label{h<t}
    If $G$ is a triangle-free graph, then $\gamma_h(G)\leq \gamma_t(G)$. 
\end{theorem}

Considering the complete $\ell$-partite graph $K_{k,\ldots, k}$ with each partite set of size $k$ ($\ell\ge 3$), we see that the difference $\gamma_h(G)-\gamma_t(G)$ can be made arbitrarily large, meaning that $\gamma_h(G)\leq \gamma_t(G)$ does not necessarily hold in general graphs $G$. They further showed that, if $G$ is a triangle-free graph of order $n$ with $\delta(G)\geq 2$, then $\gamma_h(G)\leq (\frac{1+\ln \delta(G)}{\delta(G)})n$. As shown in the above result due to Henning and Rad \cite{henning20172}, a large clique in a graph $G$ increases the hop domination number of $G$. Hence, to obtain a good upper bound on $\gamma_h(G)$, we need to impose some forbidden subgraph condition on $G$, such as triangle freeness.

Motivated by this observation, together with their results on the hop domination number of a triangle-free graph, in this paper, we focus on giving a sharp upper bound on $\gamma_h(G)$ for triangle-free graphs $G$. Our main result is the following.

\begin{theorem}\label{thm:min:degree}
Let $\mathcal{B}=\{C_4,C_7,C_8,C_{14},G_9, G_{14}, G'_{14}\}$, where the graphs $G_9,G_{14}$ and $G'_{14}$ are given in Figure~\ref{fig:BB}. Let $n$ be a positive integer with $n\ge 4$. For a connected triangle-free graph $G$ with $n$ vertices, 
if $\delta(G)\ge 2$ and $G\not\in\mathcal{B}$, then $\gamma_h(G)\le \frac{2n}{5}$.
\end{theorem}

\begin{figure}[h!]
    \centering
    \includegraphics[width=14cm]{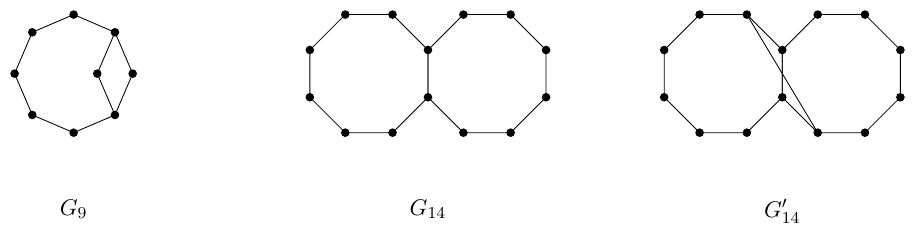}
    \caption{The graphs $G_9,G_{14}$ and $G'_{14}$ in $\mathcal{B}$}
    \label{fig:BB}
\end{figure}

Theorem~\ref{thm:min:degree} says that for a triangle-free graph $G$ with $\delta(G)\ge 2$, if $G$ has at least 15 vertices then  $\gamma_h(G) \le \frac{2n}{5}$.

Let $P_t:v_1v_2\ldots v_t$ be a path on $t$ vertices, and for each vertex $v_i$ of $P_t$, we attach an edge $v_iu_i$ to $v_i$ and then attach a $4$-cycle at the vertex $u_i$. Let $G$ be the resulting graph. Figure~\ref{fig:Tight} shows an illustration when $t=6$. For this graph $G$, we always have $\delta(G)\ge2$ and $\gamma_h(G)=2t=\frac{2|V(G)|}{5}$. Thus, the ratio $\frac{2}{5}$ in Theorem~\ref{thm:min:degree} is tight. 

We also show that the ratio $\frac{2}{5}$ can be improved up to roughly $\frac{1}{3}$ if a triangle-free graph contains a Hamiltonian path or a Hamiltonian cycle (see Corollary~\ref{newcor}). 

\begin{figure}[h!]
    \centering
    \includegraphics[width=12cm]{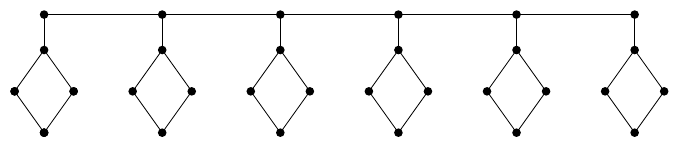}
    \caption{A graph $G$ such that $\gamma_h(G)=\frac{2|V(G)|}{5}$}
    \label{fig:Tight}
\end{figure}

Henning \cite{henning2009DM} showed that every graph $G$ with no isolated vertex satisfies the inequalities $\gamma(G)\leq \gamma_t(G)\leq 2\gamma(G)$. Hence, by Theorem~\ref{h<t}, $\gamma_h(G)\leq \gamma_t(G)\leq 2\gamma(G)$ holds for connected triangle-free graphs $G$. With this relationship in mind, what kind of relative magnitude relationships can be obtained when comparing the tight upper bounds of these three graph parameters for connected triangle-free graphs $G$ with $\delta(G)\geq 2$?
 
In \cite{Domination_MS1989}, McCuaig and Shephard showed the following.

\begin{theorem}[\cite{Domination_MS1989}]\label{thm:original}
For a connected graph $G$ with $n$ vertices, 
if  $\delta(G)\ge 2$ and $G$ is not a graph in Figure~\ref{fig:b}, then $\gamma(G)\le \frac{2n}{5}$.
\end{theorem}

\begin{remark}\label{rmk}
\rm    Note that for a graph $H_i$ in Figure~\ref{fig:b}, $\gamma(H_i)\le \frac{2n+2}{5}$. Thus, together with Theorem~\ref{thm:original}, it follows that for every connected graph $G$ with $\delta(G)\ge 2$, we have $\gamma(G)\le \frac{2n+2}{5}$.
\end{remark}

\begin{figure}[h!]
    \centering
    \includegraphics[width=14cm]{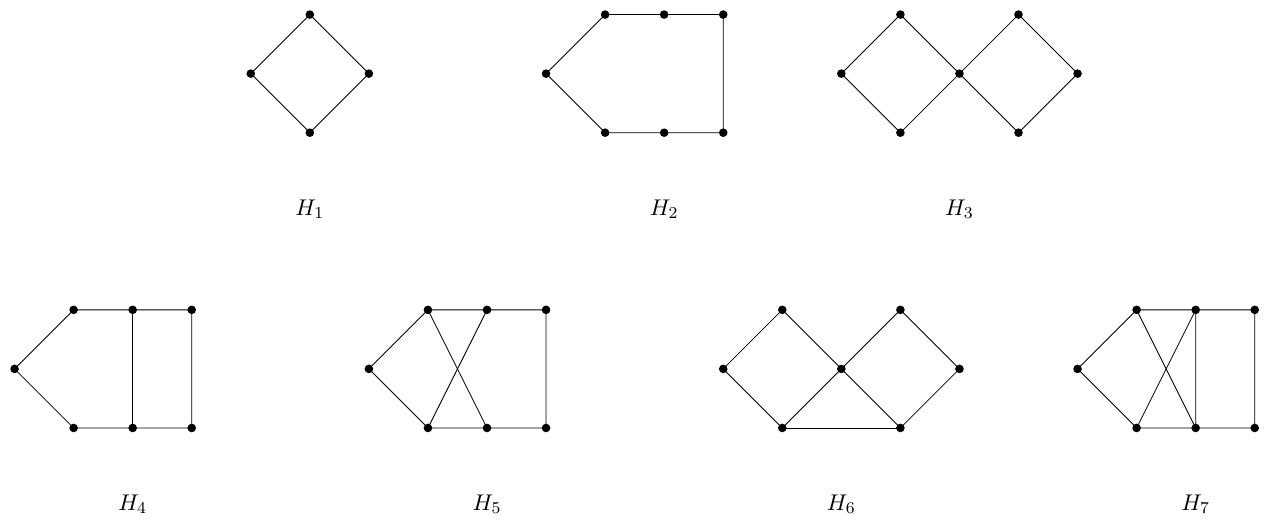}
    \caption{The graphs $H_1\sim H_7$}
    \label{fig:b}
\end{figure}

As is shown in \cite{Domination_MS1989}, there exist infinitely many connected triangle-free graphs $G$ having $\gamma(G)=\frac{2|V(G)|}{5}$ and $\delta(G)\geq 2$. 

On the other hand, Henning and Yeo \cite{henningyeoGC} gave a sharp upper bound $|V(G)|/2+\max\{1,\frac{|V(G)|}{2(g+1)}\}$ on $\gamma_t(G)$ for graphs $G$ with girth at least $g$, which means that we have the sharp upper bound $\frac{3|V(G)|}{5}$ on $\gamma_t(G)$ for the connected triangle-free graph $G$.  

Therefore, somewhat surprisingly, it can be observed from our main theorem that the tight upper bound of $\gamma_h(G)$ not only approaches but actually matches the tight upper bound of $\gamma(G)$ for connected triangle-free graphs $G$ with $\delta(G)\geq 2$. 

One might ask what happens if we relax the minimum degree condition ``$\delta(G)\geq 2$" for large connected triangle-free graphs in our work. To answer the question, we give some preliminaries.   

For a graph $G$, let $Dist(G:2)$ be the graph whose vertex set is $V(G)$ and two vertices are adjacent in $Dist(G:2)$ if and only if they have distance exactly two in $G$. 
For a vertex $v$ of a graph $G$, let $N_2(v;G)$ be the set of \textit{$2$-step neighbors} of $v$ in $G$, that is, $N_2(v;G):=\{u\in V(G)\setminus \{v\}:uv\notin E(G)$ and $N_G(u)\cap N_G(v)\neq\emptyset\}$. 
The following are some basic observations on $\gamma_h(G)$, with certain results referenced from \cite{henning20172}.

\begin{proposition}[\cite{henning20172}]\label{prop:G*} 
Let $G$ be a connected graph and $G^*=Dist(G:2)$. The following hold. 
\begin{itemize}
\item[\rm(i)] $\gamma_h(G)=\gamma(G^*)$;
\item[\rm(ii)] $N_2(v;G)=N_{G^*}(v)$ for every vertex $v$.
\end{itemize}    
\end{proposition}

Ore\cite{ore1962} showed that, any graph $G$ of order $n$ with $\delta(G)\geq 1$ satisfies $\gamma(G)\leq n/2$. 
Note that any connected triangle-free graph $G$ with $n\geq 4$ vertices satisfies $\delta(G^*)\geq 1$ unless $G\cong K_{1,t}$ for an integer $t\geq 3$. Since $\gamma_h(K_{1,t})=2$, 
combining these facts together with Proposition~\ref{prop:G*}(i), we can obtain the following theorem. 

\begin{theorem}\label{conn.version}
  Let $n$ be a positive integer with $n\ge 4$. If $G$ is a connected triangle-free graph with $n$ vertices, then $\gamma_h(G)\le \frac{n}{2}$.  
\end{theorem}

The bound on $\gamma_h(G)$ is best possible in this theorem. To see this, consider the cases $G=K_{1,3}$ or $C_8$. 
Considering the graph $G$ obtained from $K_{1,t}$ by subdividing each edge with three vertices, we see that $\gamma_h(G)=2t=\frac{|V(G)|-1}{2}$, thereby showing that the ratio $\frac{1}{2}$ of the upper bound on $\gamma_h(G)$ in Theorem~\ref{conn.version} is tight.
Since $t$ can be arbitrarily large, our main result, Theorem~\ref{thm:min:degree}, substantially refines the coefficient of the upper bound on $\gamma_h(G)$ in Theorem~\ref{conn.version}, reducing it from $\frac{1}{2}$ to $\frac{2}{5}$ for connected triangle-free graphs $G$ with $\delta(G)\geq 2$.

\section{Preliminaries}

In this section, we give further preliminaries to prove our main theorem.

\begin{proposition}\label{obs:two_cycles} 
Let $G$ be a {triangle-free} graph. Then the following hold:
\begin{itemize}
\item[\rm(i)] If  $\emptyset\neq N_G(w') = N_G(w)$ for some two vertices $w$ and $w'$ of $G$, then $\gamma_h(G)\le \gamma_h(G-w)$.
\item[\rm(ii)] For a spanning subgraph $H$ of $G$, $\gamma_h(G)\le\gamma_h(H)$. 
\end{itemize}  
\end{proposition}
\begin{proof} (i) Take a minimum hop dominating set $S$ of $G-w$. If $w'\in S$, then $w'$ is a 2-step neighbor of $w$. If $w'\not\in S$, then $w'$ has a 2-step neighbor $v$ in $S$, and so $v$ is also a 2-step neighbor of $w$. Then $S$ is a hop dominating set of $G$, which proves (i).

(ii) Take a minimum hop dominating set $S$ of $H$. Then every vertex $v$ not in $S$ has a 2-step neighbor $v'$ in $H$. Since  $v'$ is also a 2-step neighbor of $v$ in $G$ by the assumption that $G$ is a triangle-free graph, $S$ is a hop dominating set of $G$.
\end{proof}

Although it is a folklore, the domination numbers of paths and cycles were given in \cite{Domination_path_cycle}. 

\begin{theorem}[\cite{Domination_path_cycle}]\label{pathcycle}
For an integer $n\geq 3$, $\gamma(C_n)=\lceil n/3\rceil$. 
\end{theorem}

In view of Proposition~\ref{prop:G*}, note that $\gamma_h(P_n)=\gamma(P_n^*)=\gamma(P_{\lceil n/2\rceil})+\gamma(P_{\lfloor n/2\rfloor})$ and $\gamma_h(C_n)=\gamma(C_n^*)=2\gamma(C_{n/2})$ for even $n$ and $\gamma_h(C_n)=\gamma(C_n^*)=\gamma(C_{n})$ for odd $n$, where $P_n^*=Dist(P_n:2)$ and $C_n^*=Dist(C_n:2)$.
Thus, by Proposition~\ref{prop:G*} and Theorem~\ref{pathcycle}, we obtain the hop domination numbers of paths and cycles as follows.

\begin{theorem}\label{pathcyclehop}
    Let $n$ be an integer. The following statements hold:
\begin{description}
        \item{\rm{(i)}} For $n\ge 1$, $\gamma_h(P_n)=\displaystyle\left\lceil\frac{\lfloor n/2\rfloor}{3}\right\rceil+\left\lceil\frac{\lceil n/2\rceil}{3}\right\rceil$.
        \item{\rm{(ii)}} For $n\ge 3$, 
$\gamma_h(C_n)= \begin{cases}
   2\lceil\frac{n}{6}\rceil&\text{even }n,\\
   \lceil\frac{n}{3}\rceil&\text{odd }n.\\
\end{cases}$ \end{description}
\end{theorem} 

By Theorem~\ref{pathcyclehop} and Proposition~\ref{obs:two_cycles}(ii), we obtain the following corollary, and the equality of the upper bounds on $\gamma_h(G)$ can be attained when $G$ is isomorphic to a path or a cycle. So the upper bounds on $\gamma_h(G)$ are best possible in this sense. 

\begin{corollary}\label{newcor}
Let $G$ be a triangle-free graph of order $n$.  
\begin{description}
    \item{\rm{(i)}} If $G$ contains a Hamiltonian path, then $\gamma_h(G)\leq \displaystyle\lceil\frac{\lfloor n/2\rfloor}{3}\rceil+\lceil\frac{\lceil n/2\rceil}{3}\rceil$.
        \item{\rm{(ii)}} If $G$ contains a Hamiltonian cycle, then
\begin{eqnarray*}\gamma_h(G)&\leq& \begin{cases}
   2\lceil\frac{n}{6}\rceil&\text{even }n,\\
   \lceil\frac{n}{3}\rceil&\text{odd }n.
\end{cases}
\end{eqnarray*}
\end{description}
\end{corollary}

Here, we gather several inequalities and useful observations that will be used in the proofs. By Theorem~\ref{pathcyclehop}(i),
for a positive integer $n$, 
\begin{eqnarray}\label{eq:Pn} \gamma_h(P_n)&\le& \begin{cases}
   \frac{2n+6}{5}&\text{if }n=2,\\
   \frac{2n+4}{5}&\text{if }n\in\{1,3,8\},\\
\frac{2n+2}{5}&\text{if } n\in \{4,7,9,14\},\\
 \frac{2n}{5} &\text{otherwise}.\\
\end{cases}
\end{eqnarray}

\begin{proposition}\label{prop:C:two:adjacent}
For $n\ge 4$, it holds that
\begin{eqnarray}\label{eq:Cn} 
\gamma_h(C_n)&\le & \begin{cases}
   \frac{2n+2}{5} &\text{if }n\in\{4,7,14\},\\
   \frac{2n+4}{5} &\text{if }n=8,\\
   \frac{2n}{5} &\text{otherwise}.
\end{cases} \end{eqnarray}
Moreover, if $n\neq 8$, then there is a hop dominating set $S$ of $C_n$ such that $S$ contains two adjacent vertices and $|S|\le \frac{2n+2}{5}$.
\end{proposition}
\begin{proof}
By Proposition~\ref{obs:two_cycles}(ii) and \eqref{eq:Pn}, we have $\gamma_h(C_n)\le \gamma_h(P_n)\le \frac{2n}{5}$ if $n\ge 15$.
If  $n\le 14$,  Theorem~\ref{pathcyclehop}(ii) gives Table 1 for $\gamma_h(C_n)$. Thus \eqref{eq:Cn} holds.

\begin{table}[h!]\label{table1}
\centering\begin{tabular}
{c||c|c|c|c|c|c|c|c|c|c|c|c}
$n$&4 &5&6 &7&8&9&10&11&12&13&14\\ \hline
$\left\lfloor\frac{2n}{5}\right\rfloor$&1 &2&2 &2&3&3&4&4&4&5&5\\ \hline
$\gamma_h(C_n) $ &2&2&2&3&4&3&4&4&4&5&6\\
\end{tabular}
\caption{$\gamma_h(C_n)$ for small $n$}
\end{table}

If $n$ is even, then it is clear that 
{the union of dominating sets of two disjoint cycles of length $\frac{n}{2}$ is
a minimum hop dominating set $S$ of $G$,} and so there is a minimum hop dominating set $S$ of $C_n$ such that $S$ has two adjacent vertices. Hence the 'moreover' part is true when $n$ is even.

Suppose that $n$ is odd. Let $C_n: v_1v_2\cdots v_nv_1$ and $S= \{ v_i\mid i\equiv 1,2\pmod 6 \}$.
Clearly $S$ contains two adjacent vertices, and $|S|\le \frac{n+4}{3}$. Also, $S$ is a hop dominating set of $C_n$.
If $n\ge 15$, then $|S|\le \frac{n+4}{3} \le \frac{2n+2}{5}$. 
If $n\in \{7,13\}$, then $|S|=\frac{2(n-1)}{6}+1 = \frac{n+2}{3}\le \frac{2n+2}{5}$.
If $n\in\{5,11\}$, then   $|S|=\frac{2(n+1)}{6} = \frac{n+1}{3}\le \frac{2n+2}{5}$.
If $n=9$, then $|S|=4\le \frac{2n+2}{5}$. 
Thus the 'moreover' part is also true when $n$ is odd.
\end{proof}

For positive integers $a_1, \ldots, a_m$ with $3\le a_1\le \cdots \le a_m$, let $C(a_1,\ldots,a_m)$ be a graph obtained from vertex disjoint cycles $C_{a_1}$, $\ldots$, $C_{a_m}$ by identifying one vertex from each cycle. 
Note that $C(a_1)$ is a cycle of length $a_1$. See Figure~\ref{ckl} for an illustration.

\begin{figure}[h!]
    \centering
    \includegraphics[width=15cm]{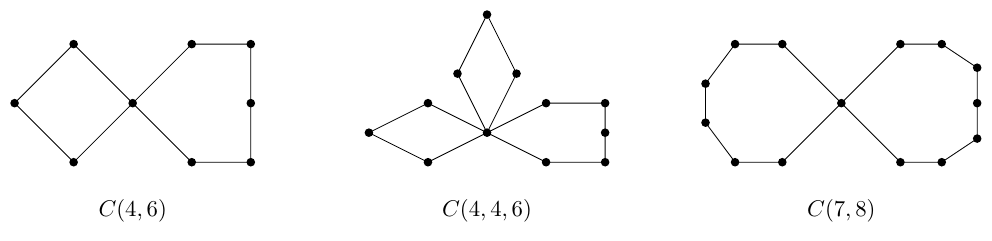}
    \caption{The graphs $C(4,6)$, $C(4,4,6)$, and $C(7,8)$}
    \label{ckl}
\end{figure}
\begin{proposition}\label{obs:ckl} 
For a triangle-free graph $G$ with $n$ vertices, if $G=C(a_1,\ldots,a_m)$ and $G\not\in\{C_4,C_7,C_8,C_{14}\}$, then 
$\gamma_h(G)\le \frac{2n}{5}$. Moreover, when $v$ 
is a common vertex of all cycles of $G$, there is a hop dominating set $S$ of $G$ such that $|S|\le \frac{2n}{5}$ and $v\in S$.
\end{proposition}
\begin{proof}  It is sufficient to show the `moreover' part.  We will show by induction on $m$. 
If $m=1$ then \eqref{eq:Cn} implies that the moreover part is true. Suppose that $m\ge 2$. Suppose that all cycles of $G$ are $8$-cycles, that is, $G=C(8,8,\ldots,8)$. Then $|V(G)|=7m+1$. Let $W$ be the set of vertices of $G$ that have distance four from $v$. 
By the construction, $|W|=m$.
Take a neighbor $v'$ of $v$.
For each vertex $x\in W$, we take one neighbor $x'$ of $x$, and then let $W'$ be the set of such vertices $x'$.
Then $S=W\cup W'\cup\{v,v'\}$ is a hop dominating set $S$ of $G$ such that $|S|=2m+2 \le \frac{2n}{5}$ and $S$ contains $v$. 

Now suppose that there is a cycle $C'$ of $G$ that is not an $8$-cycle. Let $v'$ be a neighbor of $v$ on the cycle $C'$. By the `moreover' part of Proposition~\ref{prop:C:two:adjacent}, we can take a hop dominating set $S'$ for $C'$ such that $S'$ contains both $v$ and $v'$ and $|S'|\le \frac{2|V(C')|+2}{5}$.
Let $H$ be the graph obtained from $G$ by deleting $V(C')\setminus\{v\}$. Note that $|V(H)|+|V(C')|-1=n$.
We also let $P$ be the graph obtained from $H$ by deleting $v$, the neighbors of $v$, and the $2$-step neighbors of $v$. Note that $P$ is a union of vertex-disjoint paths. Take a minimum hop dominating set $T$ of this graph $P$. 
Then $T\cup S'$ is a hop dominating set of $G$ containing the vertex $v$.

Suppose that $H\in \{C_4,C_7,C_8,C_{14}\}$. Then $m=2$.  
If $H=C_4$, then $S'$ is also a hop dominating set of $G$ and so 
$\gamma_h(G)\le |S'|= \frac{2|V(C')|+2}{5} \le \frac{2n}{5}$.
Suppose that $H\neq C_4$. Then $P$ is a path on $|V(H)|-5$ vertices.
If $H=C_7$, then  
$|T|=2$ by \eqref{eq:Pn} and therefore, 
$\gamma_h(G)\le |S'|+|T|=|S'|+2 \le \frac{2|V(C')|+12}{5} = \frac{2n}{5}$.
If $H\in\{C_8,C_{14}\}$, then  $\gamma_h(P)=|T|\le \frac{2|V(P)|+4}{5}$ by \eqref{eq:Pn}, and
so \[\gamma_h(G)\le |T|+|S'|\le 
\frac{2|V(P)|+4}{5} + \frac{2|V(C')|+4}{5} =\frac{2n}{5}.\] 
Now suppose that $H\not\in \{C_4,C_7,C_8,C_{14}\}$. 
By the induction hypothesis, $H$ has a hop dominating set $S$ of $G$ such that $v\in S$ and $|S|\le \frac{2|V(H)|}{5}$. 
Then $S\cup S'$ is a hop dominating set of $G$ containing {the vertex} $v$, which implies that
\[\gamma_h(G)\le |S|+|S'|-1\le \frac{2|V(H)|}{5}+\frac{2|V(C')|+2}{5}-1 \le \frac{2n}{5}.\] 
\end{proof}

A {\it pendent $k$-cycle} of $G$ is an induced cycle $v_1v_2 \cdots v_kv_1$ of $G$ such that $\deg_G(v_1)\ge 3$ and $\deg_G(v_i)=2$ for each $2\le i\le k$. We also call it a pendent $k$-cycle at the vertex $v_1$.

\begin{proposition}\label{lemma:pendent_cycle}
Any connected graph has a hop dominating set $S$ of a graph $G$ with $|S|=\gamma_h(G)$ such that for every a pendent $4$-cycle $v_1v_2v_3v_4v_1$ with $\deg_G(v_1)\ge 3$,   $v_1\in S$ and $N_G(v_1) \cap S\neq \emptyset$.
\end{proposition}

\begin{proof} Take a minimum hop dominating set $S$ of $G$ such that $|S|=\gamma_h(G)$. Suppose that $v_1\not\in S$ for some  pendent $4$-cycle $v_1v_2v_3v_4v_1$ of $G$.
Since $v_1$ is the only vertex which has distance exactly two from the vertex $v_3$, it follows that $v_3\in S$.
Then we let $T=(S\setminus\{v_3\}) \cup \{v_1\}$, and then $T$ is a hop dominating set of $G$ such that $|T|=\gamma_h(G)$ and $v_1\in T$.
Moreover, by considering $v_2$ and $v_3$, it follows that $N_G(v_1)\cap T\neq \emptyset$.  
\end{proof}

\section{Proof of Theorem~\ref{thm:min:degree}}
Recall that $\mathcal{B}=\{C_4,C_7, C_8, C_{14}, G_9, G_{14},G'_{14}\}$.
Suppose to the contrary that there is a  triangle-free graph {$G$} on $n$ vertices
such that $\delta(G)\ge 2$, $G\not\in\mathcal{B}$, and $\gamma_h(G)> \frac{2n}{5}$.
We choose such a graph $G$ so that 
\begin{itemize}
    \item[(1)] $n+|E(G)|$ is as small as possible, and subject to the condition (1),
    \item[(2)] the number of pendent $4$-cycles is as large as possible.
\end{itemize}
Then $G$ is a connected triangle-free graph with $\delta(G)\ge 2$, $G\not\in\mathcal{B}$, and $\gamma_h(G) > \frac{2n}{5}$, where $n=|V(G)|$.
By Proposition~\ref{obs:ckl}, $G \neq C(a_1, a_2,\ldots,a_m)$ and therefore $G$ has at least two  vertices of degree at least three.  
It is also easy to observe that $G$ is not a complete bipartite graph, since $\gamma_h(K_{s,t})\le 2$ for any positive integers $s,t$. {This also implies that $n\ge 6$.} 
Moreover, note that for every proper connected subgraph $H$ of $G$ such that $\delta(H)\ge 2$, either $H\in\mathcal{B}$ or $\gamma_h(H)\le \frac{2|V(H)|}{5}$ by the choice of $G$. From Proposition~\ref{obs:two_cycles} and \eqref{eq:Cn},   
it follows that $\gamma_h(H)\le \frac{2|V(H)|+2}{5}$ if $H\in \mathcal{B}\setminus \{ C_8\}$ and $\gamma_h(H)\le \frac{2|V(H)|+4}{5}$ if $H= C_8$. Hence, for every proper connected subgraph $H$ of $G$, 
\begin{eqnarray}\label{eq:general:H}
    \gamma_h(H) &\le&  \begin{cases}
        \frac{2|V(H)|+2}{5} & \text{if }H\neq C_8\\
        \frac{2|V(H)|+4}{5} & \text{if }H= C_8.
    \end{cases}
\end{eqnarray}

\begin{claim}\label{claim:cutedge}
For an edge $e=v_1v_2$ such that $\deg_G(v_i)\ge 3$ for each $i\in \{1,2\}$,
$e$ is a cut-edge of $G$ such that exactly one component of $G-e$ is either $C_4$ or $C_7$, and the other component of $G-e$ is not in $\mathcal{B}$.
\end{claim}

\begin{proof}
Note that  $\gamma_h(G)\le \gamma_h(G-e)$ by Proposition~\ref{obs:two_cycles} (ii), and 
 $\delta(G-e)\ge 2$ by assumption. 
If $G-e$ has no component in $\mathcal{B}$ then by the minimality of $G$ on (1), $\gamma_h(G)\le \gamma_h(G-e)\le \frac{2n}{5}$, a contradiction. 
Thus $G-e$ has a component in $\mathcal{B}$. 

\begin{subclaim}\label{subclaim:cutedge}
$e$ is a cut edge.   
\end{subclaim}
\begin{proof} Suppose to the contrary that $e$ is not a cut-edge. 
Then $G-e\in \mathcal{B}$. 
Since $G$ is triangle-free, $G-e\neq C_4$.
If $G-e=G_9$, then we can take a cycle $C:v_1v_2\ldots v_8v_1$ of length $8$ and relabel the vertices of $C$ so that $e=v_1v_s$ and $s\in\{4,5\}$, which implies that $\{v_1,v_8\}\cup \{z\}$, where $z$ is the vertex in $(V(G)\setminus V(C))$, is a hop dominating set of $G$ of size three, which is a contradiction.
Thus $G-e\neq G_9$. 
Then $G-e\in \{C_7, C_8, C_{14}, G_{14},G'_{14}\}$, and so $G-e$ has a Hamiltonian cycle $C:v_1v_2\cdots v_nv_1$ and so we can relabel the vertices so that $e=v_1v_s$ and  $4\le s\le \left\lceil\frac{n+1}{2} \right\rceil$. 
Let $H$ be a spanning subgraph of $G$ such that
$H=C+e$.
By Proposition~\ref{obs:two_cycles} (ii), it is enough to show that $\gamma_h(H)\le\frac{2n}{5}$.
Note that $|V(C)|\in\{7,8,14\}$. 
Let 
\[ S=\begin{cases}
    \{v_1,v_n\} &\text{ if }n\in\{7,8\},\\
     \{v_1,v_2,v_6,v_{9},v_{10}\} &\text{ if }n=14\text{ and }s\in\{4,5\},\\   
    \{v_1,v_2,v_5,v_{9},v_{10}\} &\text{ if }n=14\text{ and }s\in\{6,7\},\\
 \{v\in V(G)\mid \deg_G(v)\ge 3\} &\text{ if }n=14\text{ and }s=8.
\end{cases}\]
Then we can easily check that $S$ is a hop dominating set of $H$ such that $|S|\le \frac{2n}{5}$, except for the case where $n=14$ and $s=8$.
Suppose that $n=14$ and $s=8$. Then $H=G_{14}$ and $e=v_1v_8$.
For the edge $e'\in E(H)\setminus E(C)$, we may assume that $C+e'$ is also equal to $G_{14}$ (otherwise we consider $G-e'$ as $H$). 
Since $G\neq G'_{14}$, without loss of generality, we may assume that $e'\in\{ v_3v_{10}, v_4v_{11}\}$.
Then the end vertices of $e$ and $e'$ form a hop dominating set of $G$. Then $\gamma_h(G)\le 4$, which is a contradiction.
\end{proof}

By Subclaim~\ref{subclaim:cutedge}, $e=v_1v_2$ is a cut-edge, and so let $D_1$ and $D_2$ be two components of $G-e$, say $v_i\in V(D_i)$ for each $i\in\{1,2\}$.

\begin{subclaim}\label{subclaim:c4c7}
If $D_i\in \mathcal{B}$, then  $D_i$ is either $C_4$ or $C_7$.
\end{subclaim}
\begin{proof}  Let $D_2\in\mathcal{B}$.
Suppose to the contrary that  $D_2\in \{C_8, C_{14}, G_9, G_{14},G'_{14}\}$.
First, suppose that $D_2\in \{C_8,C_{14},G_{14},G_{14}'\}$.
Note that $D_2$ has a Hamiltonian cycle $C$,   let $P$ be the path obtained from $C$ by deleting the vertices of $N_C[v_2]$, and $T$ be a minimum hop dominating set of $P$.

Let $H$ be the graph obtained from $D_1$ by attaching a pendent $4$-cycle $v_1w_2w_3w_4v_1$ at the vertex $v_1$.  Clearly, $H$ is  a triangle-free connected graph,  $H\not\in \mathcal{B}$, $|V(H)|+|E(H)|<|V(G)|+|E(G)|$ and $\delta(H)\ge 2$. 
By Proposition~\ref{lemma:pendent_cycle}, there is a hop dominating set $S$ of $H$ such that $v_1\in S$, $|S|=\gamma_h(H)$, $N_H(v_1)\cap S\neq\emptyset$.
By the minimality of $G$ on (1), and so $|S|\le \frac{2|V(H)|}{5}\le \frac{2n+6-2|V(D_2)|}{5}$. 
In addition,  $S^*=(S\setminus\{w_3\})\cup T$ is a hop dominating set of $G$. (If $S$ contains $w_2$ or $w_4$, then we replace $w_4$ with $v_2$, when we consider $S\cup T$.)  
If $D_2=C_8$, then $P$ is a path on 5 vertices and so $|T|\le 2$ and $|S|\le \frac{2n-10}{5}$. 
If $D_2\in \{C_{14}, G_{14},G_{14}\}$, then $P$ is a path on $11$ vertices, $|T|\le 4$ and $|S|\le \frac{2n-22}{5}$.
Thus, in any case, $\gamma_h(G)\le  |S\cup T|\le \frac{2n}{5}$, which is a contradiction.

Suppose that $D_2=G_9$. Let $H$ be the graph obtained from $G- (V(D_2)\setminus\{v_2\})$ by attaching a pendent $4$-cycle at the vertex $v_2$. Note that $H$ is a triangle-free connected graph,  $H\not\in \mathcal{B}$, $|V(H)|+|E(H)|<|V(G)|+|E(G)|$ and $\delta(H)\ge 2$. 
By Proposition~\ref{lemma:pendent_cycle}, there is a hop dominating set $S$ of $H$ such that $v_2\in S$, $N_H(v_2)\cap S\neq \emptyset$ and $|S|=\gamma_h(H)$.
Let $x\in N_H(v_2)\cap S$.  Then $(S\setminus\{x\}) \cup \{v_1,a,b\}$ is a hop dominating set of $G$, where $ab$ is an edge of $D_2$ farthest from $v_2$.
By the minimality of $G$ on (1), $|S|=\gamma_h(H)\le \frac{2|V(H)|}{5}$ and therefore $|(S\setminus\{x\}) \cup \{v_1,a,b\}|\le \frac{2(n-5)}{5}+2=\frac{2n}{5}$, which is a contradiction.  Thus $D_2\in \{C_4,C_7\}$. Hence, the subclalim holds.
\end{proof}
If $D_1,D_2\in \mathcal{B}$, then by Sublcaim~\ref{subclaim:c4c7}, $G$ is one of the graphs in Figure~\ref{fig:c4c7}. 
{It is easy to check that $\gamma_h(G)$ is at most $\left\lfloor \frac{2n}{5} \right\rfloor$.}
Thus, exactly one of $D_1$ and $D_2$ is not in $\mathcal{B}$ and so the claim holds from Subclaim~\ref{subclaim:c4c7}.
\begin{figure}[h!]
    \centering
    \includegraphics[width=16cm]{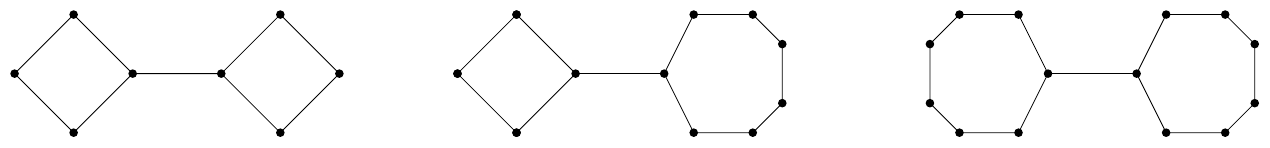}
    \caption{Some graphs}
    \label{fig:c4c7}
\end{figure}
\end{proof}

\begin{claim}\label{claim:path}
For a path $P: v_1v_2v_3v_4v_5v_6v_7$ in $G$ such that $\deg_G(v_i)=2$ for $2\le i\le 6$ and $\deg_G(v_1),\deg_G(v_7)\ge 3$, the graph $G-\{v_2,v_3,v_4,v_5,v_6\}$ is disconnected.
\end{claim}

\begin{proof}
Let $H=G-\{v_2,v_3,v_4,v_5,v_6\}$. Note that  $\delta(H)\ge 2$. Suppose to the contrary that $H$ is connected. 
If  $H\not\in \mathcal{B}$,  then there exists a hop dominating set $S$ such that $|S|\le \frac{2|V(H)|}{5} =\frac{2n}{5}-2$  by the minimality of $G$ on (1). 
Then $S\cup \{ v_3,v_4\}$ is a hop dominating set of $G$ and has size at most $\frac{2n}{5}$, which is a contradiction. Thus $H\in \mathcal{B}$.
By Claim~\ref{claim:cutedge}, $v_1v_7$ is not an edge of $G$.
 
If $H=C_4$, then $G=G_9$, which is a contradiction. Therefore $H\in \mathcal{B}\setminus\{ C_4\}$.
If $H\in\{C_7,C_8,C_{14},G_{14},G'_{14}\}$, then $H$ has a Hamiltonian cycle and so $G$ has a Hamiltonian path $P$ on $n$ vertices, where $n\in\{12,13,19\}$, and therefore, $\gamma_h(G)\le \gamma_h(P)\le \frac{2n}{5}$ by Proposition~\ref{obs:two_cycles}~(ii) and \eqref{eq:Pn}.
Suppose that $H=G_9$. Then let $w_1w_2\cdots w_8w_1$ be a cycle of $H$ and $w'_1$ be the vertex such that $N_H(w'_1)=\{w_2,w_8\}$.
Since $G\neq G'_{14}$, it follows that $\{w_1,w'_1\}\neq \{v_1,v_7\}$ and we may assume that $\deg_G(w'_1)=2$. 
If $w_1$ or $w'_1$ is in $\{v_1,v_7\}$, then we can find an edge whose ends have degree at least three, which is a contradiction to Claim~\ref{claim:cutedge}. 
Then  $N_G(w'_1)=N_G(w_1)$. Thus $G-w'_1$ is not in 
$\mathcal{B}$ and so $\gamma_h(G)\le \gamma_h(G-w'_1)\le\frac{2(n-1)}{5}$ by Proposition~\ref{obs:two_cycles} (i) and the minimality of $G$, a contradiction.
\end{proof}

Let $G^*=Dist(G:2)$. 
If $\Delta(G)= n-1$ then it is easy to see that $\gamma_h(G)\le 2$ by taking a vertex of degree $n-1$ and any other vertex. Thus $\Delta(G)\le n-2$, and so  $\delta(G^*)\ge 1$.

\begin{claim}\label{claim:G*=1}
$\delta(G^*)=1$.
\end{claim}
\begin{proof}
Suppose to the contrary that $\delta(G^*)\ge 2$.
If $G^*$ has no connected component in Figure~\ref{fig:b}, then $\gamma_h(G)=\gamma(G^*)\le \frac{2n}{5}$ by Proposition~\ref{prop:G*} (i) and Theorem~\ref{thm:original}. 
Suppose that $G^*$ has $H_i$ as a connected component, where $H_i$ is a graph in Figure~\ref{fig:b}.

\begin{subclaim}\label{subclaim:ww'}
There are no two vertices $w$ and $w'$ such that $\deg_G(w)=\deg_G(w')=2$ and $N_G(w)=N_G(w')$.
\end{subclaim}
\begin{proof}
Suppose that there are two vertices $w$ and $w'$ such that $\deg_G(w)=\deg_G(w')=2$ and $N_G(w)=N_G(w')$.  Let $N_G(w)=\{x,y\}$. 
Note that $\delta(G-w)\ge 2$, since $\delta(G^*)\ge 2$.
If $Dist(G-w;2)$ also has minimum degree at least two, then $\gamma_h(G)\le \gamma_h(G-w) =\gamma(Dist(G-w;2))\le \frac{2(n-1)+2}{5}=\frac{2n}{5}$ by Proposition~\ref{obs:two_cycles} (i) and Remark~\ref{rmk}, which is a contradiction. Thus $Dist(G-w;2)$ has minimum degree one.  Let $z$  be a vertex of degree one in $Dist(G-w;2)$. Since $z$ has degree two in $G^*$, it follows that  $N_2(z;G)=\{w,u\}$ for some vertex $u$. 
Since $N_G(w)=N_G(w')=\{x,y\}$, it follows that $u=w'$.
Since $\delta(G)\ge 2$, it also follows that $N_G(z)=\{x,y\}$. The neighbor of $x$ other than $z$ must be $w$ or $w'$, which implies that $N_G(x)= N_G(y)=\{z,w,w'\}$. Then $G=K_{2,3}$, which is a contradiction.
\end{proof}
Suppose that $i\in\{1,2,3\}$. For each edge $e$ of $H_i$, there is a common neighbor $w(e)$ of the end vertices of $e$ in $G$. 
Since $H_i$ is triangle-free, such vertices $w(e)$ are all distinct and $\deg_G(w(e))=2$.
By Subclaim~\ref{subclaim:ww'}, for each edge $e$ of $G^*$, a common neighbor $w(e)$ of the end vertices of $e$ exists uniquely.
It follows that $G$ is a subdivision of $H_i$. Then it is easy to check that either $G\in \mathcal{B}$ or $\gamma_h(G)\le \frac{2n}{5}$, a contradiction.

Suppose that $i\in\{4,5,6,7\}$. 
Then $G^*$ has a path $v_1v_2v_3v_4$ such that $\deg_{G^*}(v_2)=\deg_{G^*}(v_3)=2$, and $\deg_{G^*}(v_1), \deg_{G^*}(v_4)\ge 3$.
For each $s\in\{1,2,3\}$, since $v_s v_{s+1}$ is an edge of $G^*$, there is a common neighbor $w_s$  of $v_s$ and $v_{s+1}$ in the graph $G$. Note that since $\deg_{G^*}(v_2)=\deg_{G^*}(v_3)=2$, it follows that $w_1,w_2,w_3$ are distinct vertices, and $\deg_G(w_1)=\deg_G(w_2)=\deg_G(w_3)=2$. 
If $\deg_G(v_2)=\deg_G(v_3)=2$, 
then we reach a contradiction to Claim~\ref{claim:path} by considering the path  $v_1w_1v_2w_2v_3w_3v_4$.
We may assume that $\deg_G(v_2)\ge 3$. 
By Proposition~\ref{prop:G*} (ii), the 2-step neighbors of $v_2$ are only $v_1$ and $v_3$.  Then there is a vertex $w'$ such that either $N_G(w')=N_G(w_1)$ or $N_G(w')=N_G(w_2)$, which is a contradiction to Subclaim~\ref{subclaim:ww'}.
\end{proof}

By Claim~\ref{claim:G*=1}, $\delta(G^*)=1$ and so there is a vertex $u$ such that $\deg_{G^*}(u)=1$. Let $N_{G^*}(u)=\{v\}$ and $W=N_G(u)$. Since $\delta(G)\ge 2$, $|W|\ge 2$.  Since $N_{G^*}(u)=\{v\}$, $G$ is triangle-free and $\delta(G)\ge 2$, it follows that $N_G(w)=\{u,v\}$ for every $w\in W$. Thus $W\subset N_G(v)$. 

\begin{claim}\label{claim:W}
It holds that $v$ is a cut-vertex of $G$, and $|W|=2$.
\end{claim}

\begin{proof}
Since $G$ is not a complete bipartite graph, $N_G(v)\setminus W\neq\emptyset$, and therefore, $v$ is a cut-vertex of $G$.
Suppose to the contrary that $|W|\ge 3$. 
Then take a vertex $w\in W$, and 
note that $\delta(G-w)\ge 2$ and $G-w$ is connected. Moreover, $G-w\not\in \mathcal{B}$. By Proposition~\ref{obs:two_cycles} (i) and the minimalty of $G$, $\gamma_h(G)\le \gamma_h(G-w) \le \frac{2(n-1)}{5}$, which is a contradiction. Thus $|W|=2$. 
\end{proof}

It follows from Claim~\ref{claim:W} that $v$ has degree at least three in $G$.

\begin{claim}\label{claim:deg3}
It holds that $\deg_G(v)=3$.
\end{claim}

\begin{proof} Suppose to the contrary that $\deg_G(v)\ge 4$.
First, suppose that $v$ has a neighbor $x$ such that $\deg_G(x)\ge 3$. Then $x\in N_G(v)\setminus W$. 
By Claim~\ref{claim:cutedge}, $e=vx$ is a cut-edge of $G$ and exactly one connected component of $G-e$ is in $\{C_4,C_7\}$ and the other connected component is not in $\mathcal{B}$. Let $D_1$ and $D_2$ be the connected components of $G-e$ such that $D_1\not\in \mathcal{B}$ and $D_2\in \{C_4,C_7\}$. 
Since $\deg_G(v)\ge 4$, $x\in V(D_2)$. By the minimality of $G$ on (1), there is a hop dominating set $S_1$ of $D_1$ such that $|S_1|\le \frac{2|V(D_1)|}{5}$, $v\in S_1$  and $N_{D_1}(v)\cap S_1\neq\emptyset$ by Proposition~\ref{lemma:pendent_cycle}.
If $D_2=C_4$, then $S_1\cup \{x\}$ is a hop dominating set of $G$.
If $D_2=C_7$, then $S_1\cup \{a,b \}$ is a hop dominating set of $G$, where $ab$ is an edge of $D_2$ farthest from $x$.
Thus every neighbor of $v$ has degree two.  

Let $N_G(v)\setminus W=\{x_1,x_2,\ldots,x_t\}$. Let $y_i$ be a neighbor of $x_i$ other than $v$.
 We divide the proof into two cases.

\medskip 

\noindent (Case1) Suppose that for some $i$, $y_i$ has degree at least three. Let $H=G-x_i$. Then $\delta(H)\ge 2$. 
By Proposition~\ref{lemma:pendent_cycle}, there is a  hop dominating set $S$ of $H$ such that $\gamma_h(H)=|S|$, $v\in S$ and $S\cap N_H(v)\neq \emptyset$. Then $S$ is also a hop dominating set of $G$, and so $\gamma_h(G)\le |S|$. 
If $H$ is connected, then $|S|\le \frac{2(n-1)}{5}$.
Suppose that $H$ has two connected components $D_1$ and $D_2$. Then say $v\in D_1$ and then $D_1\not\in\mathcal{B}$. By induction, $\gamma_h(D_1)\le \frac{2|V(D_1)|}{5}$.
Let $D_2$ be the other connected component of $H$.
If $D_2\neq C_8$, then $\gamma_h(D_{2})\le \frac{2|V(D_2)|+2}{5}$ by \eqref{eq:general:H}, and hence, $\gamma_h(G)\le \gamma_h(D_1)+\gamma_h(D_2)\le \frac{2n}{5}$.
If $D_2=C_8$, then 
$\gamma_h(G)\le |S\cap V(D_1)|+|\{x_1,a,b\}|
\le \frac{2(n-9)}{5}+3 \le \frac{2n}{5}$, where $ab$ is an edge of $D_2$ that is farthest from $y_1$.
In any case, we reach a contradiction.

\medskip 

\noindent (Case 2) Suppose that $y_i$ has degree $2$  for every $1\le i\le t$. 
Let $z_i$ be the neighbor of $y_i$ other than $x_i$.
Since $G$ is triangle-free, $v,x_i,y_i,z_i$ are distinct vertices.

Suppose that for some $i$, either $\deg_G(z_i)\ge 3$, or 
$\deg_G(z_i)=2$ and $vz_i\in E(G)$. 
If  $\deg_G(z_i)\ge 3$, then let  $H=G-\{x_i,y_i\}$.  
If  $\deg_G(z_i)=2$ and $vz_i\in E(G)$, then let  $H=G-\{x_i,y_i,z_i\}$.
Then $\delta(H)\ge 2$, and so there is a hop dominating set $S$ of $H$ such that $|S|=\gamma_h(H)$, $v\in S$, and $N_H(v)\cap S\neq\emptyset$ by Proposition~\ref{lemma:pendent_cycle}. Then $S$ is a hop dominating set of $G$.
By \eqref{eq:general:H}, $\gamma_h(G)\le |S|=\gamma_h(H)\le \frac{2(n-2)+4}{5}=\frac{2n}{5}$, a contradiction. Thus $\deg_G(z_i)=2$ and $vz_i\not\in E(G)$ for every $1\le i\le t$.

Let $w_1$ be the neighbor of $z_1$ other than $y_1$.
Let  $H$ be the graph obtained from $(G-vx_1)+x_{{1}}w_1$. Then $x_1,y_1,z_1,w_1$ form a pendent $4$-cycle at the vertex $w_1$. Then $\delta(H)\ge 2$ and $H$ has no connected component in $\mathcal{B}$. Then $|V(H)|=|V(G)|$,  $|E(H)|=|E(G)|$, and  $H$ has more pendent $4$-cycles than $G$.
By the choice of $G$ on (2), $H$ has a hop dominating set $S$ of $H$ such that $|S|\le\frac{2n}{5}$, $w_1,v\in S$, $N_H(v)\cap S\neq \emptyset$,  and $N_H(w_1)\cap S\neq \emptyset$. 
We replace $x_1$ with $z_1$ if $x_1\in S$.
Then $S$ is a hop dominating set of $G$, a contradiction.
\end{proof}

By Claim~\ref{claim:deg3}, $\deg_G(v)=3$ and  
we let $x$ be the neighbor of $v$ other than $W$.

\begin{claim}
Every neighbor of $x$ other than $v$ has degree two.
\end{claim}
\begin{proof}
Suppose that $x$ has a neighbor of $y$ of degree at least three.
Suppose that $\deg_G(x)\ge 3$. 
Let $H=G-xy$. Then $\delta(H)\ge 2$. 
By Claim~\ref{claim:cutedge}, $H$ has two connected components $D_1$ and $D_2$, such that $D_1\not\in\mathcal{B}$ and $D_2\in\{C_4,C_7\}$. Note that $y\in V(D_2)$.
Take a hop dominating set $S$ of $D_1$ such that $|S|\le \frac{2|V(D_1)|}{5}$, and $v,x\in S$ by Proposition~\ref{lemma:pendent_cycle}.
If $D_2=C_4$, then $S\cup\{y\}$ is a hop dominating set of $G$.
If $D_2=C_7$, then $S\cup\{a,b\}$ is a hop dominating set of $G$, where $ab$ is an edge of $D_2$ farthest from $y$. We reach a contradiction. Thus $\deg_G(x)=2$.
Let $H=G-(W\cup\{u,v,x\})$. If $H\not\in \mathcal{B}$, then for a minimum hop dominating set $S$ of $H$, $S\cup\{v,x\}$ is a hop dominating set of $G$ whose size is at most $\frac{2n}{5}$, a contradiction. Thus $H\in\mathcal{B}$. 
The first four graphs of Figure~\ref{fig:some} show the cases where $H\in\{C_4,C_7,C_8,C_{14}\}$.
If $H\in\{G_{14},G'_{14}\}$, then  by Proposition~\ref{obs:two_cycles}~(ii),  the {fourth} graph of Figure~\ref{fig:some} shows that  $G$ has a hop dominating set of size at most $\frac{2n}{5}$.
If $H=G_9$, then take a vertex $z$ of $D_2$ that has degree two and is on a $4$-cycle of $D_2$, and then $G-z$ is equal to the  third graph of Figure~\ref{fig:some}, and so the circled vertices of the figure together with $z$ is a hop dominating set of $G$ with size five. Hence, in any case, $\gamma_h(G)\le \frac{2n}{5}$, which is a contradiction.
\begin{figure}[h!]
    \centering
    \includegraphics[width=15cm]{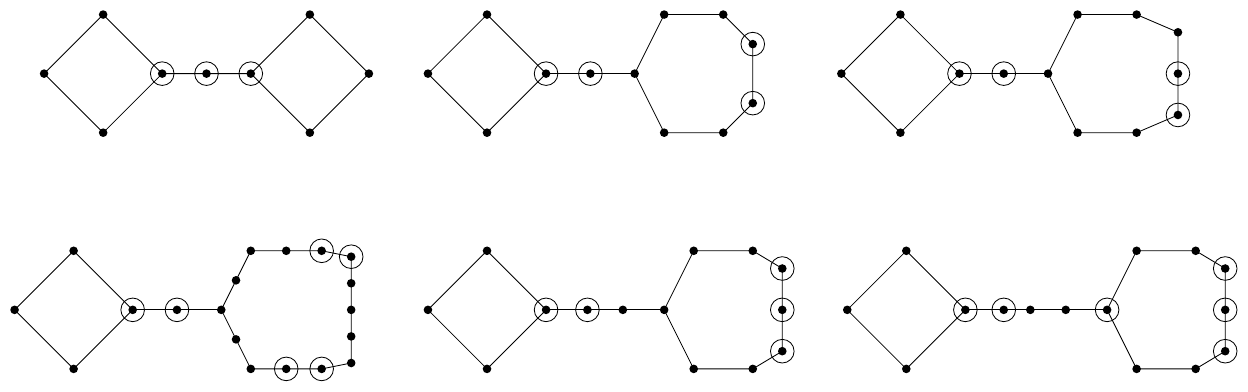}
    \caption{Some graphs, where  the circled vertices from a hop dominating set}
    \label{fig:some}
\end{figure} 
\end{proof}

Let $P:x x_1 \cdots x_k$ be a longest path of $G$ such that $\deg_G(x_i)=2$ if $1\le i<k$. Note that $\deg_G(x_k)\ge 3$.
If such path $P$ does not exists, then $G$ is obtained from some $H=C(m_1,\ldots,m_s)$ whose center vertex is $x$, by adding an edge $vx$ and a pendent $4$-cycle at vertex $v$. 
In that case, by Proposition~\ref{obs:ckl}, there is a hop dominating set $S$ of $H$, such that $x\in S$ and $|S|\le \frac{2(n-4)}{5}$, which implies that $S\cup\{v\}$ is  a hop dominating set $S$ of $G$, a contradiction. Thus such a path $P$ exists. 
Note that $k\ge 2$. 

\begin{claim}
$k\in \{2,3\}$
\end{claim}
\begin{proof}
Suppose that $k\ge 4$. Let $G'=(G-xx_1)+x_1x_4$. Note that $x_1x_2x_3x_4v_1$ is a pendent $4$-cycle at $x_4$ in $G'$. Note that $|V(G')|=n$, $|E(G')|=|E(G)|$, and $G'$ has more pendent $4$-cycles than $G$.

Suppose that $\deg_G(x)\ge 3$. Then $\delta(G')\ge 2$ and $G'$ has no connected component in $\mathcal{B}$. There is a minimum hop dominating set $S'$ of $G'$  such that $x_4,v,x\in S'$ and $N_{G}(x_4)\cap S'\neq \emptyset$ by Proposition~\ref{lemma:pendent_cycle}. By the choice of $G$, $|S'|\le\frac{2n}{5}$. Thus $S'$ is a hop dominating set of $G$, which is a contradiction.

Suppose that $\deg_G(x)=2$. Then let $H$ be the connected component of $G'$ containing $x_4$. Note that $|V(H)|\le n-5$. There is a minimum hop dominating set $S'$ of $G'$  such that $x_4\in S'$, $N_{G'}(x_4)\cap S'\neq \emptyset$ by Proposition~\ref{lemma:pendent_cycle}. If $S'$ contains $x_1$, then we replace $x_1$ with $x_3$ so that $S'$ does not have $x_1$.
Then $S'\cup\{v,x\}$ is a minimum hop dominating set of $G$, a contradiction.
\end{proof}
 
Let $G'=G-\{x_1,\ldots,x_{k-1}\}$. 
Suppose that $\deg_G(x)\ge 3$. Then $\delta(G')\ge 2$. 
If $G'$ has no connected component in $\mathcal{B}$, then 
take a minimum hop dominating set $S'$ of $G'$ such that $v,x\in S'$ by Proposition~\ref{lemma:pendent_cycle}, which implies that $S'$ is a minimum hop dominating set of $G$, a contradiction.
Thus, $G'$ has two connected components $D_1$ and $D_2$. We may assume that $D_1\not\in \mathcal{B}$ and $D_2\in\mathcal{B}$.
There is a  hop dominating set $S_1$ of $D_1$ such that $v,x\in S_1$  and $|S_1|\le \frac{2|V(D_1)|}{5}$ by Proposition~\ref{lemma:pendent_cycle}.
If $k=3$ or $D_2\neq C_8$, then   
by \eqref{eq:general:H}, there is a minimum hop dominating set $S_2$ of $D_2$ such that $|S_2|\le \frac{2|V(D_2)|+2(k-1)}{5}$.
If $k=2$ and $D_2=C_8$, then we take a minimum hop dominating set $S_2$ of a path $D_2-x_2$ such that $|S_2|=3\le \frac{2|V(D_2)|+2(k-1)}{5}$.
Note that $S_1\cup S_2$ is a hop dominating set of $G$. 
Then
\[ |S_1\cup S_2|\le \frac{2|V(D_1)|+2|V(D_2)|+2(k-1)}{5}= \frac{2n}{5} ,\] which is a contradiction.

Suppose that $\deg_G(x)=2$. Let $H$ be the connected component containing $x_k$. Then $|V(H)|\le n-6$.
If $H\neq C_8$, then for a minimum hop dominating set $S$ of $H$, we have $|S|\le \frac{2(n-6)+2}{5}$ by \eqref{eq:general:H}, and so  $S\cup\{v,x\}$ is a hop dominating set of $G$ whose size is at most $\frac{2n}{5}$, a contradiction.
Thus $H=C_8$. Then $G$ is one of the last two graphs in Figure~\ref{fig:some}, which shows that $\gamma_h(G)\le \frac{2n}{5}$. This completes the proof of our main theorem.

\section{An open problem}

It would be a natural extension to consider giving a sharp upper bound on the hop domination number for a graph with a large girth. So we propose the following open problem. 
\begin{problem}
Let $r$ be an integer with $4\leq r$. 
Determine the least value $f(n,r)$ such that the following proposition can be true: 
\begin{proposition}
If $G$ is a connected graph of order sufficiently large $n$ with $\delta(G)\geq 2$ and girth at least $r$, then $\gamma_h(G)\leq f(n,r)$. 
\end{proposition}
\end{problem}
Considering a large odd cycle, we see that $f(n,r)\geq \lceil n/3\rceil$. In this paper, we showed that $f(n,4)=\frac{2}{5}n$. For $r\geq 5$, it is open.

\section*{Acknowledgements}
Shinya Fujita was supported by JSPS KAKENHI Grant number JP23K03202.
Boram Park was supported by the National Research Foundation of Korea(NRF) grant funded by the Korea government(MSIT) (No. RS-2025-00523206) and by the framework of international cooperation program managed by the National Research Foundation of Korea (NRF-2023K2A9A2A06059347).

\bibliographystyle{abbrv}
\bibliography{ref}

\end{document}